\documentclass{article}

\usepackage[english]{babel}
\usepackage[utf8]{inputenc}

\usepackage{graphicx}
\usepackage{amsmath}
\usepackage{amsthm}
\usepackage{amssymb}
\usepackage{amsfonts}
\usepackage{float}
\usepackage{booktabs}
\usepackage{listings}
\usepackage{ae}
\usepackage{mathrsfs}

\usepackage{listings}

\usepackage{color}

\newtheorem{theorem}{Theorem}

\newtheorem{lemma}[theorem]{Lemma}

\newtheorem{definition}[theorem]{Definition}

\newtheorem{proposition}[theorem]{Proposition}
\newtheorem{corol}[theorem]{Corollary}

\newtheorem{problem}{Problem}
\newtheorem{remark}[theorem]{Remark}
\newtheorem{question}{Question}

\newtheorem*{conventions}{Conventions}

\title{$LD$-stability for Goldie rings}

\author{ Vyacheslav Futorny, Jo\~ao Schwarz, Ivan Shestakov}

\AtEndDocument{\bigskip{\footnotesize%
 
  \addvspace{\medskipamount}
 (V.\,Futorny) \textsc{Instituto de Matematica e Estatistica, Universidade de S\~{a}o Paulo, Caixa Postal 66281,\\ S\~{a}o Paulo, CEP 05315-970, Brasil and International Center for Mathematics, SUSTech, Shenzhen, China} \par
  \textit{E-mail address}: \texttt{futorny@ime.usp.br}  \par
   (J.\,Schwarz) \textsc{Instituto de Matematica e Estatistica, Universidade de S\~{a}o Paulo, Caixa Postal 66281,\\ S\~{a}o Paulo, CEP 05315-970, Brasil} \par
  \textit{E-mail address}: \texttt{jfschwarz.0791@gmail.com }  \par
   (I.\,Shestakov) \textsc{Instituto de Matematica e Estatistica, Universidade de S\~{a}o Paulo, Caixa Postal 66281,\\ Sa\~{o} Paulo, CEP 05315-970, Brasil} \par
  \textit{E-mail address}: \texttt{shestak@ime.usp.br}  \par
   \addvspace{\medskipamount}
  
}}

\date{}

\begin{document}
\maketitle

\begin{abstract}

The \emph{lower transcendence degree}, introduced by J. J Zhang, is an important non-commutative invariant
in ring theory and non-commutative  geometry strongly connected to the classical Gelfand-Kirillov transcendence degree. 
For \emph{$LD$-stable} algebras, the lower transcendence degree coincides with the Gelfand-Kirillov dimension. We show that the following algebras are $LD$-stable and compute their lower transcendence degrees:  rings of differential operators of  affine
domains,  universal enveloping algebras of finite dimensional Lie superalgebras, symplectic reflection algebras and their spherical subalgebras, finite $W$-algebras of type $A$, generalized Weyl algebras over Noetherian domain (under a mild condition), some quantum groups.
We show that the lower transcendence degree behaves well with respect to the invariants by finite groups, and 
with respect to the Morita equivalence. Applications of these results are given.

\medskip

\noindent {\bf Keywords: Lower transcendence degree, Gelfand-Kirillov transcedence degree, Gelfand-Kirillov dimension, division algebras} 
\medskip

\noindent {\bf 2020 Mathematics Subject Classification:  16P90, 16P50, 16W50, 16W22, 16U20
}
\medskip

\end{abstract}

\section{Introduction}
Throughout the paper all rings will be algebras over an arbitrary base field $\mathsf{k}$ of zero characteristic, unless specified otherwise.  For a prime Goldie ring $A$ we denote by $Q(A)$ its total quotient ring.

The transcendence degree is a very useful and important invariant in the study of commutative rings. For a finitely generated commutative algebra over a
field it coincides with the  Krull dimension. In non-commutative setting different analogs of  transcendence degree were considered in \cite{Gelfand}, \cite{Borho}, \cite{Resco1}, \cite{Resco2}, \cite{Resco3}, \cite{Stafford}, \cite{Schofield}, \cite{Schofield2}, \cite{Zhang}, \cite{Yekutieli}, \cite{Bell}. 
In particular, the classical \emph{Gelfand-Kirillov transcendence degree} (\cite{Gelfand})
  was used  to show  that the skew fields of fractions of the Weyl algebras $A_n(\mathsf{k})$
  are non-isomorphic for different $n$.

Let $A$ be a $\mathsf{k}$-algebra. Recall that the \emph{Gelfand-Kirillov dimension} of $A$, $GK \, A$, is defined as:

\[ sup_V \, lim \, sup_{ n \mapsto \infty} \frac{log (dim \, V^n)}{log \, n}, \]

where $V$ ranges over all the subframes of $A$ (finite dimensional subspaces with unity), and $V^0 = \mathsf{k}, V^i= V. V^{i-1}, i>0$.
Then the  Gelfand-Kirillov transcendence degree of $A$, $GKtdg \, A$, is defined as:

\[ sup_V \, inf_b \,  lim \, sup_{ n \mapsto \infty} \frac{log (dim \,(\mathsf{k}+ bV)^n)}{log \, n}, \]

where $V$ ranges over all the subframes of $A$ and $b$ over all regular elements of the algebra. For any commutative field $\mathsf{k}$, $GKtdg \, \mathsf{k}$  coincides with the usual transcendence degree.

In general, the Gelfand-Kirillov transcendence degree is  difficult to compute. Zhang (\cite{Zhang2}, \cite{Zhang}) developed a general technique that allows to compute $GKtdg \, A$ for several classes of algebras $A$, including certain quantum groups, semiprime Goldie PI-algebras, and certain Artin-Schelter regular algebras.
 The results of \cite{Gelfand}, \cite{Borho} and \cite{Lorenz} were also recovered.

However, in the non-commutative setting the Gelfand-Kirillov transcendence degree has some undesirable aspects.  For instance, the answers to the following very natural questions 
in general are unknown (\cite{Zhang2}, \cite{Zhang}, \cite{Bell}):

\begin{question}
Let $D \subset Q$ be division algebras. 
\begin{enumerate}
\item
 Is it true that $GKtdg \, D \leq GKtdg \, Q$?
\item
Assume $[Q:D] < \infty$, that is, $Q$ has finite dimension over $D$ (as right space). Is it true that  $GKtdg \, D = GKtdg \, Q$?
\end{enumerate}
\end{question}

 Zhang  \cite{Zhang} introduced a  concept of the  \emph{the lower transcendence degree} (henceforth denoted $LD$), for which the questions above have positive answers. 
 This invariant is connected with 
 non-commutative geometry (\cite[Section 9]{Zhang}, \cite{Stafford2}, \cite{RSS}), and ring theory (\cite[Section 9]{Zhang}, \cite{Bell}). 
 There are no known examples of division algebras for which the lower transcendence degree  differs from $GKtdg$.

Following \cite{Zhang} we say that algebra $A$ is {\em $LD$-stable} if $LD \, A = GK \, A$. 
For an $LD$-stable  prime Goldie algebra $A$ we have
 we have 
 $$LD \, A = GKtdg \, A = GKtdg \,Q(A) = LD \, Q(A) = GK \,A=GKtdg \,A_S,$$
  for any denominator set $S$ of regular elements in $A$ (cf. \cite{Zhang} p. 160).

In this paper we explore the concept of $LD$-stability and establish the following result, extending the list of $LD$-stable algebras in  \cite[Theorem 0.5]{Zhang}.

\

\begin{theorem}\label{main-1} The following algebras are $LD$-stable:
\begin{enumerate}
\item\label{item-1}
  The rings of differential operators $D(L)$ and $D(A)$, where $A$ is an affine domain and $L=Q(A)$;
\item\label{item-5}
Finite $W$-algebras of type $A $;
\item\label{item-4}
Generalized Weyl algebras  $D(a, \sigma)$ with  a Noetherian commutative domain $D$ (under a mild condition);
\item\label{item-2}
 The quantum group $U_q(\mathfrak{g})^+$ for a finite dimensional semisimple Lie algebra $\mathfrak{g}$  over an algebraically closed field and for $0 \neq q \in \mathsf{k}$ which is  not root of unity;
\item\label{item-3} $U_q(sl(N))$;
\item \label{item-7}
Universal enveloping algebras of finite dimensional Lie superalgebras;
\item\label{item-6}
Symplectic reflection algebras and their spherical subalgebras.
\end{enumerate}
\end{theorem}

Moreover, we compute the lower transcendence degree of all above algebras. 
Our second main result is the following theorem.

\begin{theorem}\label{main-2}
\begin{enumerate}
\item\label{item-21}
Let $A$ be an algebra, $G$ a finite group of automorphisms of $A$. 
If $A$ is a prime Goldie ring such that $A^G$ is also a prime Goldie ring (with some mild assumptions), then $LD \, A = LD \, A^G$. In particular, if $A$ is $LD$-stable then so is $A^G$.
\item\label{item-22}
Let $R$ and $S$ be prime Goldie rings with a prime context between them. Then $LD \, R = LD \, S$. In particular, if $R$ is Morita equivalent to $S$ and $R$ is $LD$-stable, then $S$ is also $LD$-stable.
\end{enumerate}
\end{theorem}
 
 \
 

Here is an outline of the paper.
In Section 2 we recall the definition of lower transcendence degree and the main facts about this invariant that are going to be used throughout the paper.
In  Section 3 we discuss the lower transcendence degree for rings of differential operators and prove  Theorem \ref{main-1}, (\ref{item-1})  in Proposition \ref{Prop-DiffOp} and Theorem \ref{Thm-DiffOp}. In Corollary \ref{GK} we compute explicitly the Gelfand-Kirillov transcendence degree of the quotient division ring of the ring of differential operators on any affine commutative domain, generalizing the classical result of Gelfand and Kirillov for the Weyl fields. We also give a simple proof of  the embedding of quotient division rings for rings of differential operators (Theorem \ref{embedding}).
In Section 4 we discuss how the lower transcendence degree behaves with respect to the invariants  under the action of finite groups and prove  Theorem \ref{main-2}, (\ref{item-21}) in Theorem \ref{inv4}. Here Proposition \ref{inv1} and Corollary \ref{inv2} are of independent interest.
 In Section 5 we discuss Galois algebras and generalized Weyl algebras. We make an assumption on the embedding into a skew group ring, which is the case in  essentially all known examples of Galois algebras (cf. \cite{Futorny}). We give a lower bound for the Gelfand-Kirillov dimension of Galois algebras (Proposition \ref{Galois0}) and prove  Theorem \ref{main-1}, (\ref{item-5}) in Proposition \ref{Galois3}. Then we discuss generalized Weyl algebras  and prove  Theorem \ref{main-1}, (\ref{item-4}) in Proposition \ref{GWA}.  Finally, we discuss the $LD$-stability of Ore extensions of the polynomial algebra in Proposition \ref{Alev}. 
Section 6 is devoted to the case of quantum groups and Lie superalgebras. We prove  Theorem \ref{main-1}, (\ref{item-2}, \ref{item-3}, \ref{item-7}) in Theorems \ref{quantum1},  \ref{quantum2} and Proposition \ref{superalgebras}.
In Section 7 we  consider the prime context  between two rings, in particular for Morita equivalence  (Proposition \ref{moritaisprime}), and prove  Theorem \ref{main-2}, (\ref{item-22}) in Theorem \ref{primecontexts}.
In Section 8, we consider symplectic reflection algebras. First we show that the spherical subalgebra is $LD$-stable (Theorem \ref{spherical}), and then we exhibit a prime context between the symplectic reflection algebra and its spherical subalgebra (Proposition \ref{primecontext-sympletic}), thus showing that the  symplectic reflection algebra is also $LD$-stable (Theorem \ref{sympletic}). This proves  Theorem \ref{main-1}, (\ref{item-6}).

 \begin{conventions}
All relevant algebra invariants are considered over the fixed basis field $\mathsf{k}$. The same  applies to the tensor product of algebras.
\end{conventions}

\section{Lower transcendence degree} 
 We recall  the main properties of the notion of lower transcendence degree. We refer to \cite{Zhang} for motivations for this  subtle invariant.

Let $A$ be an algebra and $V$ a subframe of it. If for any such $V$ there exists a finite dimensional non-zero subspace $W \subset A$ such that $dim \, VW = dim \, W$, then we define $LD \, A$ as 0. Otherwise, there exists a subframe $V$ such that for every finite dimensional non-zero $W$:

\[ dim \, VW \geq dim \, W + 1. \]

For a real number $d>0$ we say that $V$ satisfies $VDI(A)_d$, {\em the volume difference inequality} (see \cite{Zhang}),  if there exists  $c>0$ such that for every finite dimensional non-zero $W$:

\[ dim \, VW \geq dim \, W + c (\dim \, W)^{(d-1)/d}. \]

If instead, we have:

\[ dim \, VW \geq dim \, W + c \dim \, W, \]

then we say that $V$ satisfies $VDI(A)_\infty$.

\begin{definition}
The lower transcendence degree of an algebra $A$ is either 0 or:

\[ LD \, A = sup_V sup \{d \ | \ VDI(A)_d \, \text{holds for } V \}, \]

where $V$ runs through all subframes of $A$.
\end{definition}

We shall call the \emph{$LD$-degree}  the value of the lower transcendence degree. The following theorem summarizes some of its most important properties.

\begin{theorem} \label{main-Zhang}\cite[Theorem 0.3, Proposition 0.4, Proposition 2.1, Proposition 3.1]{Zhang} 
We have:
\begin{enumerate}
\item If $Q \subset D$ are division algebras such that $[D:Q]$ is finite (dimension as a right space), then $LD \, Q = LD \, D$.
\item If $A$ is any algebra and $S$ any denominator set of regular elements, then $LD \, A = LD \, A_S = LD \, _S A $. In particular, if $A$ is a prime Goldie, then $LD \, A = LD \, Q(A)$.
\item For any algebra $A$, $LD \, A \leq GKtdg \, A \leq GK \, A$; the equality holds if $A$ is PI and prime. In particular, for a commutative field it coincides with the usual transcendence degree.
\item Let $A$, $B$ be prime Goldie algebras. If $B \subset A$ then $LD \, B \leq LD \, A$. If $A$ is a finitely generated right $B$-module and $B$ is Artinian, then $LD \, A = LD \, B$.
\end{enumerate}
\end{theorem}



\section{$LD$-stability for rings of differential operators}
In this section we will discuss rings of differential operators and  embeddings of their quotient division rings.

\begin{definition}
Let $A$ be a commutative $\mathsf{k}$-algebra. Define inductively $D(A)_0=A$ and $D(A)_n = \{ d \in End_\mathsf{k} \, A | [d,a] \in D(A)_{n-1}, \forall a \in A \}$, and finally $D(A)=\bigcup_{i=0}^\infty D(A)_i$. This way we obtain a natural structure of filtered associative $\mathsf{k}$-algebra. If $A$ is affine and regular, $D(A)$ coincides with the subring of $End_\mathsf{k} \, A$ generated by $A$ and the module $Der_\mathsf{k} \, A$ of $\mathsf{k}$-derivations. In the later case, it is well known that the ring is a simple Noetherian domain (\cite[Chapter 15]{McConnell}).
\end{definition}

In particular,  the Wey algebras $A_n(\mathsf{k})$ are the rings of differential operators on the polynomial algebras, with generators $x_1, \ldots, x_n, y_1, \ldots , y_n$ and relations $[x_i, x_j] = [y_i, y_j] =0$, $[y_i, x_j]=\delta_{ij}, i,j=1,\ldots,n$.

\

The following simple lemma will be used frequently.

\

\begin{lemma} \label{main-lemma}
If $A$ and $B$ are two prime Goldie rings,  $A \subset B$, $LD \, A = LD \, B$ and $B$ is $LD$-stable, then so is $A$.

\begin{proof}
By Theorem \ref{main-Zhang}(3)(4), $LD \, A \leq GK \, A \leq GK \, B = LD \, B = LD \, A$.
\end{proof}
\end{lemma}

\

Let $B$ be an affine commutative domain over $\mathsf{k}$ (not necessarily regular) with the field of fractions $L$  and $n = trdeg \, B$.

\

\begin{proposition} \label{Prop-DiffOp}
$D(L)$ is $LD$-stable with the $LD$-degree $2n$.
\begin{proof}
By \cite[15.2.5, 15.3.10]{McConnell}, $D(L)$ is an Ore domain with the $GK$ dimension $2n$.
 It contains a copy of the Weyl Algebra $A_n(\mathsf{k})$. Since the Weyl Algebra is $LD$-stable (\cite[Theorem 0.5]{Zhang}), we have by Theorem \ref{main-Zhang}(3)(4)
  $$2n = GK \,  A_n(\mathsf{k}) =  LD \, A_n(\mathsf{k}) \leq LD \, D(L) \leq GK \, D(L) = 2n.$$
\end{proof}
\end{proposition}

\

\begin{theorem} \label{Thm-DiffOp}
$D(B)$ is $LD$-stable and the $LD$-degree is $2n$.

\begin{proof}
We can realize $D(B)$ as a subset of $D(L)$ in the following way (\cite[15.5.5(iii)]{McConnell}):
\[ D(B) = \{ d \in D(L)|d(B) \subset B \}. \]
The ring
$D(L)$ is a non-commutative domain with finite Gelfand-Kirillov dimension. Since $D(B)$ is a subring of $D(L)$, the same  holds for it. Hence, $D(B)$ does not contain a subalgebra isomorphic to the free associative algebra in two variables. It follows by \cite[Prop. 4.13]{Krause} that $D(B)$ is an Ore domain.
By \cite[Proposition 1.8]{Muhasky}, $Q(D(B))=Q(D(L))$, and hence $LD \, D(B) = LD \, D(L)$ by Theorem \ref{main-Zhang}(2).  Since $D(L)$ is $LD$-stable by Proposition \ref{Prop-DiffOp} then $D(B)$ is $LD$-stable 
 by Lemma \ref{main-lemma}.
\end{proof}
\end{theorem}
\

With this we have shown  Theorem \ref{main-1}(\ref{item-1}). It has the following important corollary:
\

\begin{corol} \label{GK}
$GKtdg \, Q(D(B)) = 2n$.
\end{corol}

This is a  generalization of the well known fact   that $GKtdg \, Q(A_n(\mathsf{k})) = 2n$ \cite{Gelfand}.

Now we are going to discuss the question of embeddings of quotient division rings of algebras of differential operators. Recall the following fact obtained, by different reasonings, by Joseph (\cite{Joseph}), Resco (\cite{Resco1}), and Bavula (\cite{Bavula}):

\begin{itemize}
\item
Let $X$ and $Y$ be two affine irreducible smooth varieties over an algebraically closed field. Denote $D(X):= D(\mathcal{O}(X))$ and similarly $D(Y)$, $\mathcal{O}(,)$ the ring of regular functions. If $dim \, X > dim \, Y$, then there is no $\mathsf{k}$-algebra embedding of $Q(D(X))$ into $Q(D(Y))$.
\end{itemize}

Applying Theorem \ref{Thm-DiffOp} and Theorem \ref{main-Zhang} we immediately obtain the following generalization of this fact.

\

\begin{theorem} \label{embedding}
Let $A$ and $B$ be affine integral domains with transcendence degrees $n$ and $m$, respectively. If $n>m$ then there is no $\mathsf{k}$-algebra embedding of $Q(D(A))$ into $Q(D(B))$.
\end{theorem}

\


\section{$LD$-stability for rings of invariants}

In this section we discuss  the lower transcendence degree of the subalgebra of  invariants  under the action of a finite group. 
We start with the following observation.

\

\begin{proposition} \label{inv1}
Let $A$ be an Ore domain and $G$ a finite group of automorphisms of $A$. Then $LD \, A^G = LD \, A$.
\end{proposition}

\begin{proof}
Set $Q:=Q(A)$. Then  $Q(A^G)=Q^G$ (\cite{Faith}), and we have 
$LD \, A= LD \, Q$ and $LD \, A^G =LD \, Q^G$ by Theorem \ref{main-Zhang}(2). 
Also,  $[Q : Q^G ] \leq |G|$ by the Noncommutative Artin's Lemma (\cite[Lemma 2.18]{Montgomery}), so we get $LD(Q)=LD(Q^G)$ by Theorem \ref{main-Zhang}(1),  and the result follows.
\end{proof}

\

Applying Proposition \ref{inv1} and Lemma \ref{main-lemma} we obtain

\begin{corol} \label{inv2}
Let $A$ be an Ore domain and $G$ a finite group of automorphisms of $A$. If $A$ is $LD$-stable then so is $A^G$.
\end{corol}

\

A similar result holds in more general situations. We recall the following fact 

\begin{proposition} \label{inv3}\cite[Theorem 1.15, Corollary 5.9]{Montgomery} 
If $R$ is a semisimple Artinian algebra and $G$  a finite group of automorphisms of $R$, then $R^G$ is also semisimple Artinian. Moreover, $R$ is a finitely generated left and right module over $R^G$.
\end{proposition}

\

We now present the main result of this section, which is the content of  Theorem \ref{main-2} (\ref{item-21}).

\begin{theorem} \label{inv4}
 Let $A$ be a prime Goldie ring,  $G$  a finite group of automorphisms of $A$. Suppose that $G$ has the nondegenerate trace on $A$, or that the ring $A$ has no nilpotent elements. Suppose also that $B=A^G$ is a prime Goldie ring (see, e.g., \cite[Theorem 3.17]{Montgomery}). Then $LD \, A = LD \, B$, and if $A$ is $LD$-stable then so is $B$.
 \end{theorem}
 
\begin{proof}
 Set $Q=Q(A)$ and $P=Q(B)$.
Then 
  $Q^G = P$ by \cite[Theorem 5.3]{Montgomery}. Moreover, both $Q$ and $ P$ are simple Artinian by Goldie's Theorem, and hence they are prime Goldie rings. By Proposition \ref{inv3} and Theorem \ref{main-Zhang}(4), we have $LD \, P = LD \, Q$. Hence, $LD \, A = LD \, B$ by Theorem \ref{main-Zhang}(2). Finally, the last claim follows from Lemma \ref{main-lemma}.
\end{proof}

\

\section{$LD$-stability for Galois algebras }
In this section we apply the results of the previous sections for the lower transcendence degree for Galois algebras, introduced in \cite{FO1}, \cite{FO2}, cf. also \cite{Futorny}.
The setup of Galois algebras is as follows. Let $\mathsf{k}$ be algebraically closed. Consider a pair $\Gamma \subset U$ of algebras, where $\Gamma$ is a finitely generated over $\mathsf{k}$ commutative domain, and  $U$ is an associative algebra finitely generated over $\Gamma$.
 Let $K := Q(\Gamma)$, and $L$ a finite Galois extension of this field with the Galois group $G$. Let $\mathfrak{M} \subset Aut_\mathsf{k} \, L$ be a monoid of automorphisms with the following $\emph{separating}$ property: if $m,m' \in \mathfrak{M}$ have the same restriction to $K$, then they coincide. Finally, we assume that $G$ acts on $\mathfrak{M}$ by conjugations.

\begin{definition}[\cite{FO1}]
If there is an embedding of $U$ into the invariant skew monoid ring $\mathcal{K}:=(K*\mathfrak{M})^G$ such that $KU=UK=\mathcal{K}$, then  $U$ is a \emph{Galois algebra} over $\Gamma$.
\end{definition}

We have

\begin{proposition} \label{main-Galois}\cite[Proposition 4.2]{FO1}
Let $S = \Gamma - \{ 0 \}$. Then $S$ is a left and right denominator set for $U$, and localization (on both sides) by $S$ gives us an isomorphism $U_S \cong \mathcal{K}$.
\end{proposition}

In all known cases of Galois algebras, $\mathfrak{M}$  is an ordered semigroup, and in many cases $\mathfrak{M}=\mathbb{Z}^n$, $n \geq 1$  (see \cite{Futorny} and references therein). Note  that $\mathbb{Z}^n$ is an ordered group by \cite[13.1.6]{Passman}.
Henceforth we make the following assumption:

\

$(\dagger)$ $\mathfrak{M}$ is isomorphic to $\mathbb{Z}^n$ for adequate $n$.

\

Under this condition we have:

\

\begin{corol} \label{corol-Galois}
$\mathcal{K}$ and $U$  are Ore domains.
\end{corol}

  \begin{proof}
Since $\mathbb{Z}^n$ is an ordered group, we have that $K* \mathbb{Z}^n$ is an Ore domain, and hence so is its invariant subring $\mathcal{K}$ (\cite{Faith}). By Proposition \ref{main-Galois}, the same holds for $U$.
\end{proof}

As our first application of the lower transcendence degree to Galois algebras, we recall the following result  \cite[Theorem 6.1]{FO1}:

\begin{itemize}
\item Under certain technical assumption, $GK \, U \geq GK \, \Gamma + Growth(\mathfrak{M})$
(cf.  \cite[12.11]{Krause}).
\end{itemize}

Under the condition $(\dagger)$, we can show this result in a straightforward way.

\

\begin{proposition} \label{Galois0}
Let $U$ be a Galois algebra in $\mathcal{K}$ with $\mathfrak{M} \cong \mathbb{Z}^n$. Then 
$$GK \, U \geq GK \, \Gamma + n.$$
\end{proposition}

\begin{proof}
To begin with, it is clear that $Growth(\mathbb{Z}^n)$ is $n$ (cf. \cite{FO1} p. 627). Also, Corollary \ref{corol-Galois}, Theorem \ref{main-Galois} and Theorem \ref{main-Zhang}(1) imply that $LD \, U = LD \, \mathcal{K}$. It is also true that $LD \, \mathcal{K} = LD \, K*\mathfrak{M}$ by Proposition \ref{inv1}. By \cite[Theorem 4.5]{Krause},  $GK \, \Gamma = trdeg \, K$, which equals, as we saw,  to $LD \, K$. Applying  \cite[Corollary 5.5, 2]{Zhang} we obtain  
$$LD \, K*\mathfrak{M} \geq trdeg \, K +n = GK\, \Gamma +n.$$  Hence, $GK \, U \geq LD \, U = LD \, \mathcal{K} = LD \, K*\mathfrak{M} \geq GK \, \Gamma +n$.
\end{proof}

\

In some cases we can establish the $LD$-stability and compute the $GK$ dimension. The follow proposition shows how to obtain this  under some assumptions on the Galois algebra. In particular, it proves  Theorem \ref{main-1} (\ref{item-5})

\

\begin{proposition}\label{Galois3} Let $U$ be a Galois algebra in $\mathcal{K}$ with $\mathfrak{M} \cong \mathbb{Z}^n$.
If $GK \ U = GK \, \Gamma +  n$, then  $U$ is $LD$-stable. In particular, all finite $W$-algebras of type $A$ are $LD$-stable.
\end{proposition}

\begin{proof}
The first claim is clear, as we have shown that $LD \, U = GK \, \Gamma +n$ above;  the second statement is \cite[Theorem 3.3]{FMO}.
\end{proof}


Next we consider the generalized Weyl algebras introduced by Bavula \cite{Bavula0}. They play important role in  noncommutative algebraic geometry in small dimensions and have numerous applications in representation theory, homological algebra and ring theory (cf. \cite{BY} for  details).

\begin{definition}
Let $D$ be a ring, $\sigma=(\sigma_1,\ldots, \sigma_n)$ an $n$-tuple of commuting automorphisms of $D$: $\sigma_i \sigma_j = \sigma_j \sigma_i, \, i,j=1,\ldots,n$. Let $a=(a_1,\ldots, a_n)$ be 
an $n$-tuple of non-zero divisors in the center of $D$, such that $\sigma_i(a_j)=a_j, j \neq i$. The generalized Weyl algebra $D(a, \sigma)$ of rank $n$ is generated over $D$ by $X_i^+, X_i^-$, $i=1,\ldots,n$ subject to the  relations: 

\[ X_i^+ d = \sigma_i (d) X_i^+; \, X_i^- d= \sigma_i^{-1}(d) X_i^-, \, d \in D, i=1, \ldots , n , \]
\[ X_i^-X_i^+ = a_i; \, X_i^+X_i^- = \sigma_i(a_i), \, i=1 ,\ldots , n \, ,\]
\[ [X_i^-,X_j^+]=[X_i^-,X_j^-]=[X_i^+,X_j^+]=0 \, , i \neq j.\]
\end{definition}

We will only consider the case when $D$ is a finitely generated commutative domain. In this case $D(a, \sigma)$ is also a Noetherian domain. The Weyl algebra is a natural example of a generalized Weyl algebra (\cite{Bavula0}).

It is not  true that all generalized Weyl algebras are Galois algebras, but under some very mild restriction on $\sigma$,   it is the case (cf. \cite[Theorem 14]{FS}). 

We have:
\

\begin{proposition}\label{prop-embeds}\cite[Proposition 13, 15]{FS}
$D(a, \sigma)$ allways embeds into $D*\mathbb{Z}^n$, where the canonical generators of $\mathbb{Z}^n$ act on $D$ as $\sigma_1, \ldots, \sigma_n$. Both rings have the same quotient ring of fractions.
\end{proposition}

\

Under very mild assumptions, we can show that the generalized Weyl algebras are $LD$-stable, and compute the lower transcendence degree. Namely:

\begin{proposition}\label{GWA}
Let $D(a, \sigma)$ be a generalized Weyl algebra of rank $n$ such that for every finite dimensional vector space $U \subset D$ there is another finite dimensional vector space $V \supset U$ such that $\sigma_i(V)=V, i=1, \ldots,  n$. Then $D(a,\sigma)$ is $LD$-stable, and the value of the invariant is $GK \, D + n$.
\end{proposition}

\begin{proof}
By  Proposition \ref{prop-embeds},  \cite[Corollary 5.5.2]{Zhang} and Theorem \ref{main-Zhang}(2), we get $LD \, D(a,\sigma) \geq GK \, D +n$. On the other hand,   $GK \, D(a,\sigma) = GK \, D +n$ by \cite[Corollary 3.5]{Mo}. So both invariants coincide.
\end{proof}

\

This proves  Theorem \ref{main-1}(\ref{item-4}). Many interesting generalized Weyl algebras satisfy the condition of  Proposition \ref{GWA} (cf. \cite{Mo}). In particular, we have

\

\begin{corol}\label{corol-GWA}
Every Noetherian generalized down-up algebra, the $q$-Heisenberg algebra in three generators, $\mathcal{O}_{q^2}(so(\mathsf{k}))$ and every rank $1$ quantum generalized Weyl algebras (with non-root of unity parameter) are $LD$-stable.
\end{corol}

\begin{proof}
Follows from \cite[Proposition 4.3, 4.5, Example 4.9, Proposition 4.11]{Mo}. 
\end{proof}

We finish this section with an analysis of $LD$-stability of Ore extensions of the polynomial algebra.

\begin{proposition}\label{Alev}
Let $A=\mathbb{C}[x][y; \alpha, \delta]$ ($\alpha$ an automorphism, $\delta$ an $\alpha$-derivation) be an iterated Ore extension such that the center of $Q(A)$ is $\mathbb{C}$. Then $A$ is $LD$-stable with $LD \, A = 2$ if and only if the automorphism $\alpha$ is locally algebraic, that is every $t \in \mathbb{C}[x]$ is contained in a finite dimensional $\alpha$-stable subspace. 
\end{proposition}

\begin{proof}
It follows from \cite[Théorème 3.7]{Alev1} that under the hypothesis of the theorem, $Q(A)$ is either the first Weyl field or $Q(\mathbb{C}_q[x,y])$, for $q \neq 0, 1$ and $q$ not a root of unity. Hence, $LD \, A = LD \, Q(A) = 2$ (\cite[Theorem 0.5]{Zhang}). Now $GK \, A$  equals $2$ if and only if $\alpha$ is locally algebraic by \cite[Theorem 12.3.3]{Krause}.
\end{proof}

\begin{remark}
By the results of \cite{Alev1}, if the center is bigger than the base field, then $A$ is either the polynomial algebra in two indeterminates, or a quantum plane at a root of unity, or the first quantization of the Weyl algebra at a root of unity. In all these cases $A$ is a $PI$-algebra, and hence already known to be $LD$-stable.
\end{remark}

\

\section{$LD$-stability for quantum groups and universal enveloping algebras of Lie superalgebras}

We will always assume that $q \in \mathsf{k}$ is not a root of unity. Let $\mathfrak{g}$ denote a finite dimensional semisimple Lie algebra over an algebraically closed field of characteristic 0. For $0 \neq q \in \mathsf{k}$, we consider the quantized envelopping algebra $U_q(\mathfrak{g})$ given by the quantum Chevalley-Serre relations (we refer  to \cite{Goodearl}, I.6.3, for standard definition). 

Fix a  basis $\Delta =\{ \alpha_1, \ldots, \alpha_n \}$ of the associated root system of $\mathfrak{g}$ (with $N$ positive roots). Let $w_0$ be the longest element of the Weyl group and $w_0 = s_{i_1} \ldots s_{i_N}$  a fixed reduced expression of  $w_0$ with respect to the basis $\Delta$, where $s_i$ is the reflection associated to $\alpha_i$. Order all positive roots as follows
$$\beta_1 = \alpha_{i_1}, \beta_2 = s_{i_1} \alpha_{i_2}, \ldots, \beta_N = s_{i_1} \ldots s_{i_{N-1}} \alpha_N.$$

Define the elements  $E_{\beta_j} = T_{i_1} \ldots T_{i_{j-1}} E_{i_j}$ (c.f. \cite[I.6.7, I.6.8]{Goodearl}). For $m=(m_1, \ldots, m_N) \in \mathbb{N}^N$ set $E_\beta^m = E_{\beta_1}^{m_1} \ldots E_{\beta_N}^{m_N}$. Then  a basis of $U_q(\mathfrak{g})^+$ is given by $ \{ E_\beta^m, m \in  \mathbb{N}^N \}$. 

Recall the  Levendorskii-Soibelman relations:

\[ E_{\beta_i} E_{ \beta_j} - q^{(\beta_i, \beta_j)} E_{\beta_j} E_{ \beta_i} = \sum_{m \in \mathbb{N}^N} z_m E_\beta^m, \]

where $1 \leq i < j \leq N$, $z_m \in \mathbb{Q}[q^{\pm 1}]$ and it is zero unless $m_r=0$ for $r \leq i, r \geq j$ (\cite[Proposition 1.6.10]{Goodearl}).


\begin{definition}
A total ordering $\preceq$ on the monoid $\mathbb{N}^m$ is called linear admissible if:
\begin{itemize}
\item
$x \preceq y$ implies $x+z \preceq y + z, \forall x,y,z \in \mathbb{N}^m$
\item
(0, \ldots, 0) is the smallest element.
\end{itemize}
\end{definition}

\begin{definition}
A multi-filtration of an algebra $A$ is a family $\mathcal{F} = \{ F_x(A)| x \in \mathbb{N}^m \}$ of subspaces such that:
\begin{itemize}
\item
$F_x(A) \subset F_y(A)$ if $x \preceq y$;
\item
$F_x(A)F_y(A) \subset F_{x+y}(A)$;
\item
$\bigcup_{x \in \mathbb{N}^m} F_x(A) = A$; 
\item
$1 \in F_0(A)$,
\end{itemize}
where $\preceq$ is a linear admissible total ordering on $\mathbb{N}^m$.
\end{definition}


Given the basis described above for $U_q(\mathfrak{g})^+$, we can order $\mathbb{N}^N$ lexographically and get an admissible ordering $\preceq$. Also, define a multi-filtration $\mathcal{F}_m= \{ E_\beta^p | p  \preceq m\}$, $m \in \mathbb{N}^N$.
This multi-filtration is finite dimensional and the associated graded algebra is the quantum affine space with generators $E_{\beta_i}, i=1 ,\ldots, N$ and relations 
$$ E_{\beta_i} E_{ \beta_j} = q^{(\beta_i, \beta_j)} E_{\beta_j} E_{ \beta_i},$$
 following the Levendorskii-Soibelman relations (cf. \cite[I.6.11]{Goodearl}).

The quantum affine spaces are $LD$-stable (\cite[Corollary 6.3(1)]{Zhang}), and since the multi-filtration is finite dimensional, one can use \cite[Theorem 2.8]{Torrecillas} 
to conclude that $GK \, U_q(\mathfrak{g})^+ = GK \, gr_\mathcal{F}  U_q(\mathfrak{g})^+$ is the same as the Gelfand-Kirillov dimension of the quantum affine space, which is known to be $N$ \cite[II.9.6-9]{Goodearl}.

Applying \cite[Theorem 4.3]{Zhang} we get  Theorem \ref{main-1} (\ref{item-2}):

\

\begin{theorem}\label{quantum1}
$U_q(\mathfrak{g})^+$ is $LD$-stable, and $LD \, U_q(\mathfrak{g})^+$ equals the number of positive roots of $\mathfrak{g}$.
\end{theorem}

\

Now we consider $U_q(sl_N)$ and the extended quantum group $U_q^{ext}(sl_N)$, cf. \cite[7.1]{FH}.  For the rest of  this section we assume that the base field is $\mathbb{C}$.

\

\begin{theorem}\cite[Theorem 7.1]{FH}
$$Q(U_q^{ext}(sl_N)) \cong Q( \mathbb{K}_q[x,y]^{\otimes \ N-1} \otimes \mathbb{K}_{q^2}^{\otimes \, (N-1)(N-2)/2}),$$ where $\mathbb{K}_q[x,y]$ is the quantum plane, and $\mathbb{K}=\mathbb{C}(Z_1, \ldots, Z_{N-1})$, a purely transcedental extension.
\end{theorem}

This is the quantum analogue of the Gelfand-Kirillov Conjecture for $sl_N$ (cf. \cite[I.2.11, II.10.4]{Goodearl}). Since the quantum plane has the $GK$ dimension $2$ over its base field  and since it is $LD$ stable, as we saw above, we can immediately compute  $LD \, U_q^{ext}(sl_N)$ as  (*)
$$GK \, \mathbb{K}_q[x,y]^{\otimes \ N-1} \otimes \mathbb{K}_{q^2}^{\otimes \, (N-1)(N-2)/2} = GK_\mathbb{C} \, \mathbb{K} + 2N-2 +N^2 -3N +2= N^2-1,$$ using \cite[Proposition 3.11, 3.12]{Krause}.

On the other hand, by \cite{Obul}, $GK \, U_q(sl_N)= N^2-1$. Since $U_q(sl_N)$ embeds into the extended quantum group (\cite[7.1]{FH}), $LD \, U_q(sl_N) \leq N^2-1$ by (*). Hence, by Theorem \ref{main-Zhang}(3) we get

\

\begin{theorem} \label{quantum2}
$U_q(sl_N)$ is $LD$-stable and $LD \,U_q(sl_N)$ equals  $N^2-1$.
\end{theorem}

\

This is the content of Theorem \ref{main-1} (\ref{item-3}). We now briefly discuss the $LD$-stability of the universal enveloping algebras of finite dimensional Lie superalgebras (Theorem \ref{main-1} (\ref{item-7})).

\begin{proposition}\label{superalgebras}
Let $\mathfrak{g} = \mathfrak{g}_0 \oplus \mathfrak{g}_1 $ be a finite dimensional Lie superalgebra, $\mathfrak{g}_0$ the even part, $\mathfrak{g}_1$ the odd. Then $U(\mathfrak{g})$ is $LD$-stable, and  its LD degree is $dim \, \mathfrak{g}_0$.

\begin{proof} The algebra
$U(\mathfrak{g})$ is  finitely generated and free as a left and a right module over $U(\mathfrak{g}_0)$. By \cite[Proposition 3.1]{Zhang}, $LD \, U(\mathfrak{g})  = GK \,  U(\mathfrak{g}_0) = dim \, \mathfrak{g}_0$.
\end{proof}
\end{proposition}

\section{Prime Contexts and lower transcendence degree}

In this section we will discuss how the lower transcendence degree behaves with respect to certain contexts between prime Goldie rings. In particular, we show that $LD$-degree is a Morita invariant of those rings.
We begin by recalling the following well known result.

\begin{proposition} \label{mg}
Let $_{S}M_R$ be an $S-R$ bimodule, $S=End_R(M_R)$, and $M$ be finitely generated projective on the right. Let $_{R}M^*_S=hom(M_R, R_R)$, which is finitely generated projective (left) $R$-module. Consider the $R-R$ bimodule map $M^* \otimes_S M \mapsto R$. If $m \in M$ is different from $0$ then $M^* m \neq 0$.
\end{proposition}

\begin{proof}
By the dual basis lemma (\cite[Lemma 5.2]{McConnell}),  there is a finite collection $\{ m_i \}_{i=1}^n \subset M$, $\{g_1 \}_{i=1}^n \subset M^*$ such that, for any $x \in M$, $x= \sum m_i g_i(x)$. Hence, if $m \neq 0$ then  $g_i(m) \neq 0$ for some $i$.
\end{proof}

We refer to \cite[1.1.6]{McConnell} and \cite[3.12 ]{Jacobson} for generalities on Morita contexts and recall the fundamental theorem of Morita theory.
\

\begin{theorem} \label{Morita}\cite[3.12, 3.14]{Jacobson} 
Let $R$ and $S$ be two Morita equivalent rings. Then we have a Morita context


\[ \begin{bmatrix}  R & M^* \\  M & S  \end{bmatrix}, \]

where $M$ is a bimodule $_{S}M_{R}$, $S = End(M_R)$ and $_{R}M^*_S$ is isomorphic as bimodule to both $hom(M_R, R_R)$ and $hom(_{S}M,_{S}S)$. We also have an isomorphism of bimodules between $_{S}M_{R}$ and $hom(M^*_S, S_S)$, and $hom(_{R}M^* ,  _{R}R)$; $_{S}M$,$ _{R}M^*$, $M_R$, $M^*_S$ are all finitely generated projective modules.
\end{theorem}

\

We now recall the notion of prime contexts introduced in \cite{Nicholson}  (cf. also \cite[3.6.]{McConnell})

\begin{definition} \label{defprime}
\cite[3.6.5]{McConnell}
Let $R$, $S$ be prime rings, $V$ a $R-S$ bimodule and $W$ a $S-R$ bimodule. Then the Morita context


\[ \begin{bmatrix}  R & V \\ W & S  \end{bmatrix}, \]

is called a \emph{prime context} if for $v \neq 0 \in V, s \neq 0 \in S, w \neq 0 \in W$, $vW$, $Vw$ and $VsW$ are all non-zero.
\end{definition}

\

Morita equivalences are a particular case of prime contexts, as the following result shows.

\begin{proposition} \label{moritaisprime}
If $R$ and $S$ are prime rings, then a Morita context, which is a Morita equivalence for them, is also a prime context.
\end{proposition}

\begin{proof}
Follows from Theorem \ref{Morita}  and Proposition \ref{mg}.
\end{proof}

\

The main result of this section is the following:

\begin{theorem} \label{primecontexts}
Let $R$, $S$ be prime Goldie rings with a prime context between them. Then $LD \, R = LD \, S$. In particular, this is the case if $R$ and $S$ are Morita equivalent. Since $GK$ dimension is also Morita invariant (\cite[8.2.9(iii)]{McConnell}) then Morita equivalent rings are $LD$-stable simultaneously.
\end{theorem}

\begin{proof}
 Given two prime Goldie rings $R$, $S$ and a prime context between them, $Q(R)$ is Morita equivalent to $Q(S)$ by  \cite[3.6.9]{McConnell}. Call $Q=Q(R)$, $P=Q(S)$. By the  Morita theory, $P=End(M_Q)$ for some finitely generated  $Q$-module, and $Q$ is simple Artinian by Goldie's Theorem. Hence, every module over $Q$ is completely reducible, $Q = \bigoplus_{i=1}^nL$ is a direct some of copies of some simple right $Q$-module $L$, which is the unique simple $Q$-module up to isomorphism. Set $D=End(L_Q)$. Then $D$ is a division ring and $Q \cong \mathbb{M}_n(D)$ by Wedderburn-Artin's Theorem. We have $LD \, Q = LD \, D$ by \cite[Corollary 3.2(3)]{Zhang}. Moreover, $M_Q$ is a direct sum of a finite number of $L$'s, and so $P$ is also matrix ring over $D$, and we can apply the previous reasoning. We conclude that   $LD \, P = LD \, D = LD \, Q$. Hence, $LD \, R = LD \, S$ by Theorem \ref{main-Zhang}(2).
\end{proof}

\

With this,  Theorem \ref{main-2} (\ref{item-22}) is proved.

\section{Symplectic reflection algebras}

In this section we are going to use the results of the previous section to study symplectic reflection algebras. The question of wether they are Morita equivalent or not is a very subtle one (\cite{Bellamy}). But we show that, nonetheless, they always belong to a prime context, and  are $LD$-stable.


\begin{definition}\cite{Etingof2}
Let $V$ be a  complex vector space of dimension $2n$, with a non-degenerate skew-symmetric form $\omega$. Let $\Gamma$ be a finite subgroup of $SP_{2n}(\mathbb{C})$ generated by \emph{symplectic reflections}, that is, by the elements $g \in \Gamma$ such that $1-g$ has rank two. Then $\Gamma$ is called a \emph{finite symplectic reflection group}.
\end{definition}

The data $(V, \omega, \Gamma)$ is called a {\em symplectic triple}. We also assume the triple to be indecomposable, that is we assume that $V$ can not be expressed as a direct sum of two non-trivial $\Gamma$-invariant subspaces $V_1, V_2$ with $\omega(V_1, V_2)=0$.
Let $W$ be a complex reflection group acting on a vector space $h$ and hence on its dual $h^*$. Then, for $V = h \oplus h^*$  define a bilinear form

\[ \omega((y,f),(u,g)) = g(y) - f(u), y,u \in h, \, f,g \in h^*. \]

With the diagonal action of $W$ we get then an indecomposable triple, which subsumes the case of rational Cherednik algebras.

For each symplectic reflection $s \in \Gamma$ let $\omega_s$ be the form with radical $ker \, (1-s)$ such that $\omega_s=\omega$ on $im \, (1-s)$.  Let $S$ be the set of symplectic reflections in $\Gamma$, and $c: S \mapsto \mathbb{C}$  a complex valued function in the set of symplectic reflections invariant under conjugation in $\Gamma$. Fix $t \in \mathbb{C}$.

\begin{definition}
\emph{(Symplectic reflection algebras)}
Consider the tensor algebra on $V$, $T(V)$, with the natural action of $\Gamma$ extended from that on $V$. The symplectic reflection algebra, henceforth denoted $\mathbf{H}_{t,c}$, is the quotient of the skew group ring $T(V)*\Gamma$ by the relation

\[ [x,y] = t \omega(x,y) + \sum_{s \in S} c(s) \omega_s(x,y) s, x,y \in V. \]
\end{definition}

\begin{remark}
We will assume $t \neq 0$ since otherwise the symplectic reflection algebra is $PI$ (cf. \cite{Brown}), and hence the question of $LD$-stability is already settled  in \cite{Zhang}. Nonetheless, everything that follows holds without this assumption.
\end{remark} 

\

We recall the following result which shows that the symplectic reflection algebras are prime Goldie rings.

\begin{theorem} \label{prime}\cite[Theorem 4.4]{Brown}
Symplectic reflection algebras are prime Noetherian algebras.
\end{theorem}

Let $e = 1/|\Gamma| \sum_{h \in \Gamma} h$ be an idempotent. Define the spherical subalgebra of the symplectic reflection algebra as $\mathbf{U}_{t,c}:= e \mathbf{H}_{t,c} e $. The same argument as in \cite[Theorem 4.4]{Brown} shows that it is an Ore domain and $e$ is its unit.

For symplectic reflection algebras we have a finite dimensional filtration $\mathcal{F}$ given by $F_{-1}=0$, $F_0 = \mathbb{C} \Gamma$, $F_1 = \mathbb{C} \Gamma \oplus \mathbb{C} \Gamma V, F_i=F_1^i, i \geq 2$. This filtration clearly induces a filtration (also denoted by $\mathcal{F}$) on $\mathbf{U}_{t,c}$: $\mathcal{F} \cap \mathbf{U}_{t,c}$. We have the following isomorphisms \cite[Theorem 1.3]{Etingof2}:

\

\begin{itemize}
\item
$gr_\mathcal{F} \mathbf{H}_{t,c} \cong S(V) * \Gamma$;
\item
$gr_\mathcal{F} \mathbf{U}_{t,c} \cong S(V)^\Gamma$.
\end{itemize}

This shows, in particular, that the spherical has a finite dimensional filtration whose associated filtered algebra is a domain, and hence we have:

\

\begin{theorem} \label{spherical}
The algebra $\mathbf{U}_{t,c}$ is $LD$-stable, and the value of the lower transcendence degree is $dim \, V$.
\end{theorem}

\begin{proof}
 The algebra $S(V)^\Gamma$, being a commutative domain, is $LD$-stable. By \cite[Theorem 4.5(a)]{Krause} and \cite[8.2.9]{McConnell}, $GK \, S(V)^\Gamma = dim \, V$.  Hence, $GK \, \mathbf{U}_{t,c}= dim \, V$ by \cite[Proposition 6.6]{Krause}. Now the statement follows from  \cite[Theorem 4.3(4)]{Zhang}.
\end{proof}

\

Let us now  find a prime context between the symplectic reflection algebra and its spherical subalgebra. To simplify  the notation we set $\mathbf{H}:=\mathbf{H}_{t,c}$ and  $\mathbf{U}:= \mathbf{U}_{t,c}$. Recall the following result 
 of Etingof and Ginzburg. 

\

\begin{theorem} \label{etingof}\cite[Theorem 1.5]{Etingof2}
Lets consider the right $\mathbf{U}$-module $\mathbf{H}e$, and the left $\mathbf{U}$-module $e\mathbf{H}$.
\begin{enumerate}
\item
We have an isomorphism $e\mathbf{H} \cong Hom_\mathbf{U}(\mathbf{H}e, \mathbf{U})$, where $x \in e\mathbf{H}$ goes to the map $F_x(y)=xy$, for $y \in \mathbf{H}e$.
\item
In an analogue fashion, $\mathbf{H}e \cong Hom_\mathbf{U}(e\mathbf{H}, \mathbf{U})$. In particular, the modules $\mathbf{H}e$, $e\mathbf{H}$ are reflexive.
\item
The left action of $\mathbf{H}$ on $\mathbf{H}e$  induces an isomorphism $\mathbf{H} \cong End_\mathbf{U}(\mathbf{H}e)$.
\end{enumerate}
\end{theorem}

\

Now we can show our desired prime context:

\begin{proposition} \label{primecontext-sympletic}
\[ \begin{bmatrix}  \mathbf{U} & e\mathbf{H} \\ \mathbf{H}e & \mathbf{H}  \end{bmatrix}, \]
is a prime context.
\end{proposition}

\begin{proof}
Let $0 \neq x \in e \mathbf{H}$. By Theorem \ref{etingof}, (1), it induces a non-zero homomorphism in $Hom_\mathbf{U}(\mathbf{H}e, \mathbf{U})$, given by left multiplication by $x$. So, indeed, $x \mathbf{H}e \neq 0$. Similarly, if $0 \neq y \in \mathbf{H}e$, then $e \mathbf{H} y \neq 0$, by Theorem \ref{etingof} (2). Finally, if $0 \neq s \in \mathbf{H}$ then $s \mathbf{H}e \neq 0$ by
Theorem \ref{etingof}, (3), and hence
$e\mathbf{H} s \mathbf{H}e$ is non-zero by the preceeding reasoning.
\end{proof}

\
Combining  Theorems \ref{primecontexts}, \ref{spherical} and Proposition \ref{primecontext-sympletic} we immediately obtain:

\begin{theorem} \label{sympletic}
$\mathbf{H}_{t,c}$ is $LD$-stable, and the value of the lower transcendence degree is $dim \, V$.
\end{theorem}

With this, Theorem \ref{main-1}(\ref{item-6}) is proved.

 
\




\section*{Acknowledgments}

V.\,F.\ is supported in part by the CNPq (304467/2017-0) and by the Fapesp (2018/23690-6); J.S. is supported by the Fapesp (2018/18146-5);  I. S. is supported by the CNPq (304313/2019-0) and by the Fapesp (2018/23690-6).
 The second author is grateful to K. Goodearl for useful discussion.

\end{document}